\documentclass[amscd,amssymb,verbatim,12pt]{amsart}
\usepackage{epsfig}
\usepackage{graphicx}
\newif\ifAMS
\IfFileExists{amssymb.sty}
  {\AMStrue\usepackage{amssymb}}
  {\usepackage{latexsym}}
  \usepackage{enumerate}
\theoremstyle{plain}

\newtheorem*{Besicovitch}{Besicovitch's lemma}

\newtheorem{Thm}{Theorem}[section]

\newtheorem{Lem}[Thm]{Lemma}

\theoremstyle{definition}
\newtheorem{Def}{Definition}

\theoremstyle{remark}

\newtheorem{Qu}[Thm]{Question}

\newcommand{\interior}{^{ \kern-5pt ^\circ}}

\begin{document}
\title
{Contracting thin disks}

\author
{Panos Papasoglu }


\email {} \email {papazoglou@maths.ox.ac.uk}

\address
{Mathematical Institute, University of Oxford, 24-29 St Giles',
Oxford, OX1 3LB, U.K.  }

\address
{ }

\begin{abstract} We answer a question of Liokumovich-Nabutovsky-Rotman showing
that if $D$ is a Riemannian 2-disc with boundary length $L$, diameter $d$ and area $A\ll d$ then
$D$ can be filled by a homotopy $\gamma _t$ with $|\gamma _t|$ bounded by $L+2d+O(\sqrt A)$.

\end{abstract}
\maketitle
\section{Introduction}

Gromov in \cite{Gr} asked whether for any Riemannian 2-disc of boundary length $L$ and diameter $d$ there
is a homotopy contracting the boundary to a point such that the lengths of the intermediate curves in the homotopy
are bounded by a linear function $f(d,L)$. This is related to a (still open) question of geometric group theory on the
relationship between the diameter and the area of a van-Kampen diagram.

Frankel and Katz \cite{FK} showed that for any $n\in \mathbb N$ there is a 2-disc with $L=1$ and $d=1$ such that
for any homotopy contracting the boundary to a point there is some intermediate curve of length greater than $n$, so they
answered Gromov's question in the negative. They asked further whether one can give a positive answer to Gromov's question
if one asks for a linear function depending also on the area of the disc. 

Liokumovich, Nabutovsky and Rotman \cite{LNR} gave a positive answer to this and obtained quite precise bounds on the
lengths of curves of `optimal' contracting homotopies in terms $d,L$ and the area of the disc, $A$. However they note that their bound 
does not appear to be best possible when $d$ is fixed and $A\to 0$ (the case of `thin' discs). More precisely they show that
there are homotopies $\gamma _t$ such that $|\gamma _t|$ is bounded by $2L+2d+O(A)$ and they ask if this can be improved
to $L+2d+O(A)$. We show here that this is indeed the case. 

I would like to thank A. Nabutovsky for pointing out that a problem I posed in a previous version of this paper was already solved.


%


\section{Contracting thin discs}

If $X$ is a metric space and $\gamma :[0,1]\to X$ is a path we denote by $|\gamma |$ the length of $\gamma $.
If $c\in X$ and $r>0$ we denote by $B(c,r)$ the closed ball with center $c$ and radius $r$.

We recall a definition from \cite{LNR}:

\begin{Def} Let $D$ be a Riemannian 2-disc. The 
\textit{path diastole} of $D$ is: \smallskip

$pdias(D) = \sup _{p,q\in \partial D}\, \inf _{(\gamma _t)}\, |\gamma _t | $ \smallskip

\noindent where $\gamma _t$ runs over all families of paths from $p$ to $q$, $\gamma _t: [0, 1]\to D $ with $\gamma _t(0) = p$, $\gamma _t(1) = q$,
and $\gamma _0\cup -\gamma _1=\partial D$.
\end{Def}

We will use a result from \cite{LNR} (theorem 1.6, B,C):

\begin{Thm} \label{diastole}Let $D$ be a Riemannian disc with diameter $d$. The following inequality holds:

$$pdias (D)\leq 2|\partial D|+686 \sqrt {Area (D)}+2d .$$
\end{Thm}

We recall also (\cite{Be}) :

\begin{Besicovitch}\label{L:Besicovitch}  Let $D$ be a Riemannian 2-disc and let $\gamma =\partial D$. Suppose
$\gamma $ is split into 4 subpaths, $\gamma =\alpha _1\cup \alpha
_2 \cup \alpha _3 \cup \alpha _4$. Let $d_1=d(\alpha _1, \alpha
_3)$, $d_2=d(\alpha _2, \alpha _4)$
 Then
$$Area(D)\geq d_1d_2$$
\end{Besicovitch}

We state now our main result:

\begin{Thm} Let $D$ be a Riemannian disc with diameter $d$, boundary length $L$ and area $A$.
Then for any point $x\in \partial D$ there is a homotopy based at $x$, $\gamma _t \,$ $(t\in [0,1])$
contracting $\partial D$ to $x$ such that for all $t$
$$ |\gamma _t|\leq L+2d+1000\sqrt A $$

\end{Thm}

\begin{proof}
We will prove this by `induction' on the length $L$ of $\partial D$.
Indeed if $L\leq 40\sqrt A$ then by theorem \ref{diastole} there is a homotopy $\gamma _t$ contracting $\partial D$
with $$|\gamma _t|\leq L+2d+726 \sqrt A $$ so the theorem holds.

We assume now that the theorem holds if $L\leq k\sqrt A$ for some $k\in \mathbb N, k\geq 40$.
Let $D$ be a disc with boundary length $L$ satisfying: $$(k+1)\sqrt A\geq L> k\sqrt A\, .$$ We have the following lemma:

\begin{Lem} \label{induct} There is a subpath $l_1$ of $\partial D$ with $|l_1|\leq 30\sqrt A$ 
and a path $\alpha $ in $D$ with the same endpoints as $l_1$ such that $$|\alpha |\leq |l_1|-\sqrt A$$
We may assume further that $x\notin l_1$.
\end{Lem}
\begin{proof} Consider all subpaths $l$ of $\partial D$ which satisfy the following conditions:

1. $|l| \geq 30\sqrt A$.

2.  There is a path $\beta $ in $D$ joining the endpoints of $l$ with $|\beta |\leq 5\sqrt A$. 

3. $x\notin l$.

Clearly the set of such paths is non-empty, take e.g. $\partial D$ minus a small neighborhood of $x$.

Let $l_0$ be a path of minimal length with properties 1,2,3. Let $a,b$ be the endpoints of $l_0$ and let
$l_1$ be a subpath of $l_0$ with endpoints $a,c$ and length $15\sqrt A$. 
We consider $K=B(c, 5\sqrt A)\cap \partial D$. We claim that if $y\in K$ then $y\in l_0$. Indeed otherwise there is
a path of length $\leq 10\sqrt A$ joining the endpoints of $l_1$ so the assertion of the lemma holds.
It follows by the minimality of $l_0$ that if $y\in K$ then the subpath $[y,c]$ of $l_1$ has length $\leq 30\sqrt A$.
However if the length of $[y,c]$ is greater than $6\sqrt A$ the lemma clearly holds.

\begin{figure}[h]
\includegraphics{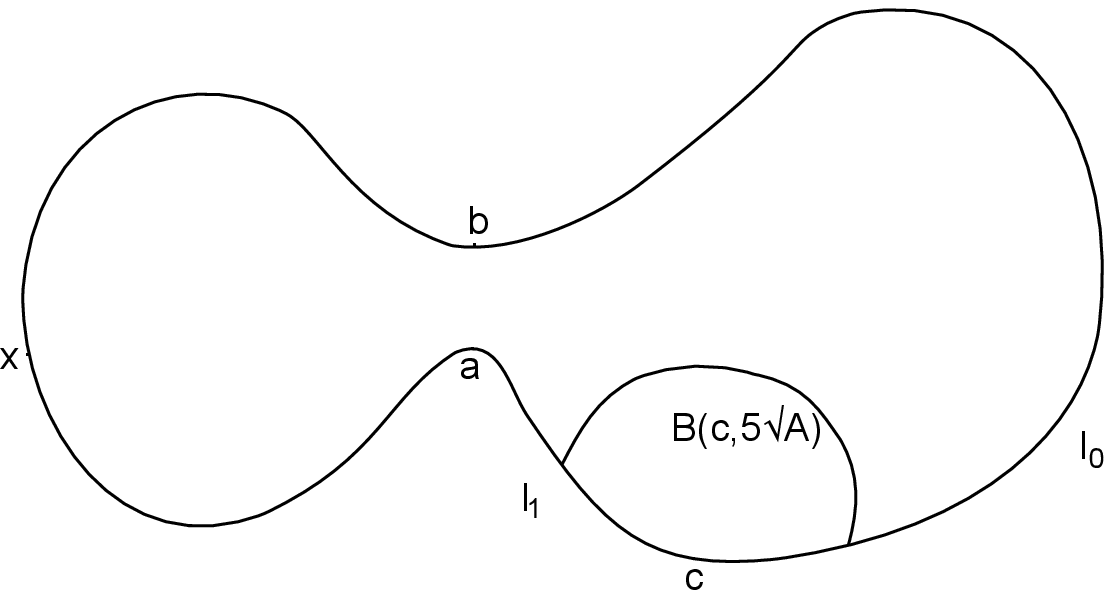}
\vspace{0.1in}
 \caption{}
\end{figure}

We conclude that $K$ is contained in a subpath $\beta_1=[b_1,b_2]$ of $l_0$ with length $\leq 12 \sqrt A$ and with
$b_1,b_2\in K$. Given $\epsilon >0$ small it is easy to see that there is a path $\beta _2:[0,1]\to D$ in 
$B(c, 5\sqrt A+\epsilon)\setminus B(c, 5\sqrt A-\epsilon)$ with $b_1=\beta _2(0),b_2=\beta _2(1)$.

Let $$t_1=\sup \{t\in [0,1]:d(\beta _2(t), [b_1,c]\leq 2 \sqrt A \}\, ,$$
$$t_2=\inf \{t\in [0,1]:d(\beta _2(t), [c,b_2]\leq 2 \sqrt A \}.$$

Let $q_1$ be  a path of length $2\sqrt A$ joining $\beta _2(t_1)$ to some point $c_1\in [b_1,c]$ and let
 $q_2$ be  a path of length $2\sqrt A$ joining $\beta _2(t_2)$ to some point $c_1\in [c,b_2]$.

We consider now the following simple arcs:

$$\alpha _1=[c-3\sqrt A, c+3\sqrt A],\, \alpha _2=[c_1,c]\cup q_1$$

$$\alpha _3= \beta_2([t_1,t_2]),\, \alpha _4=[c,c_2]\cup q_2\,.$$

We claim that $d(\alpha _1, \alpha _3)\geq 2\sqrt A-\epsilon $. Indeed if $p$ is a path joining some
$z\in \alpha _3$ to a point on $\alpha_1$ then $d(c,z)\leq |p|+3\sqrt A $ so $|p|\geq 2\sqrt A-\epsilon $
since $z\notin B(c, 5\sqrt A-\epsilon)$.

Note also that $d(\alpha _2, \alpha _4)\geq \sqrt A$. Indeed if not then there is a path of length less than $ 5 \sqrt A$ in $D$
joining the points $c_1,c_2$. However in this case
the lemma clearly holds as the length of $[c_1,c_2]$ is greater than $6\sqrt A$.

\begin{figure}[h]
\includegraphics[width=5.0in ]{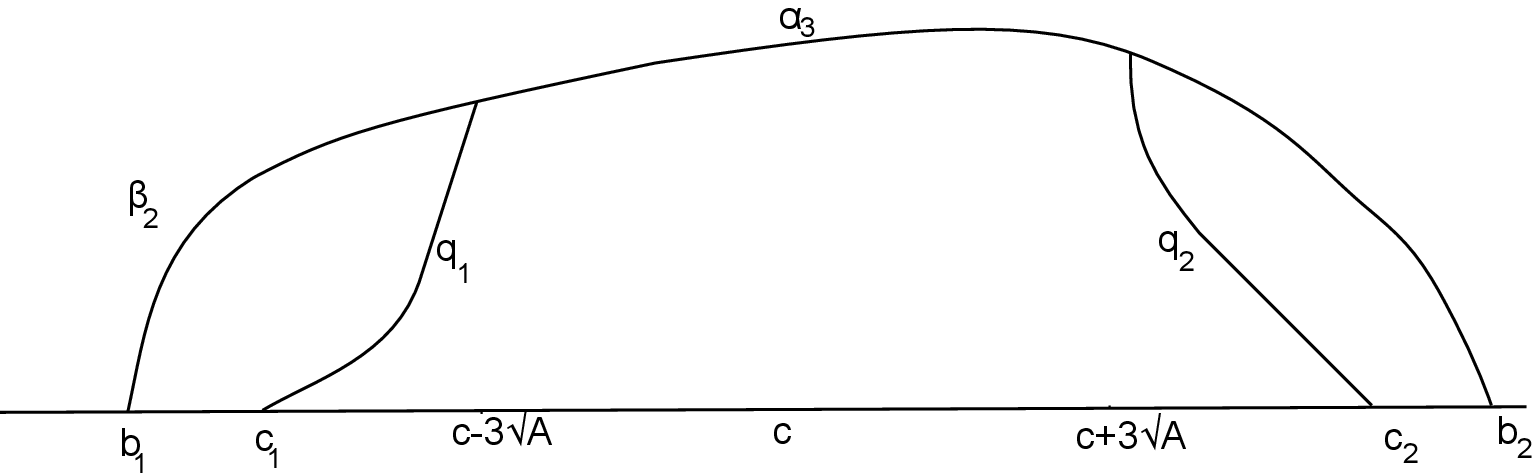}
\vspace{0.1in}
 \caption{}
\end{figure}

Applying Besicovitch's lemma we have that the area enclosed by the four simple arcs $\alpha _1, \alpha _2, \alpha _3, \alpha _4$
is greater than $$(2\sqrt A-\epsilon)\sqrt A $$ which is greater than $A$ provided that $\epsilon $ is sufficiently small, a contradiction.
\end{proof}

We proceed now with the proof of the theorem. Let $l_1$ be a subpath of $\partial D$ as in the lemma and let $\alpha $ be a path in $D$ joining
$a,b$ as in the lemma. By theorem \ref{diastole} there is a path homotopy $\beta _t$ between $l_1$ and $\alpha $ such that
$$|\beta _t|\leq 2\cdot 59 \sqrt A+686 \sqrt A+2d $$
Let's say that $\partial D=l_1\cup -l_2$
Clearly we can extend $\beta _t$ to a homotopy $\gamma _t$ based at $x$ such that $\gamma _0=\partial D, \gamma _1=-l_2\cup \alpha $ and

$$|\beta _t|\leq |\partial D|+118 \sqrt A+686 \sqrt A+2d \leq 2d+L+1000\sqrt A .$$
As $$|l_1|+|\alpha |\leq k\sqrt A$$ we can extend $\gamma _t$ to a homotopy contracting the boundary to $x$ and satisfying the
same bound on the length of the intermediate curves.

\end{proof}


%
%

\section{Discussion}

%
%
%

One wonders whether the result of Liokumovich-Nabutovsky-Rotman \cite{LNR} can be extended to higher dimensions. In particular we may ask:

\begin{Qu} Let $B$ be a Riemannian 3-ball of diameter $d$, boundary area $area(\partial B)=A$ and volume $v$. Is it true that
there is a homotopy $S_t$ contracting $\partial B$ to a point such that $area (S_t)\leq f(d,A,v)$ for some function $f$?
\end{Qu}

We may ask for a weaker result too, namely if there is always a `small' disk splitting the ball into two `big' pieces:

\begin{Qu} Let $B$ be a Riemannian 3-ball of diameter $d$, boundary area $area(\partial B)=A$ and volume $v$. Is it true that
there is a disc $D$ splitting $B$ in two regions of volume $>v/4$  such that $area (D)\leq h(d,A,v)$ for some function $h$?
\end{Qu}

We remark that a similar problem in dimension 2 is quite well understood.
Balacheff-Sabourau \cite{BS} showed that there is some $c>0$ such that any Riemannian surface $M$ of genus $g$ can be separated in two domains of equal area by a 1-cycle of length bounded by $c\sqrt {g+1}\sqrt {\rm{area}( M)}$. Liokumovich on the other hand showed \cite{Li} that given $C>0$ and a closed surface $M$ there is a Riemannian metric of diameter 1 such that any 1-cycle splitting it into two regions of equal area has length greater than $C$. Finally Liokumovich-Nabutovsky-Rotman \cite{LNR} show that for any $\epsilon >0$ a disc (or a sphere) $D$ of area $A$ can be subdivided into two regions of area $>\dfrac{1}{3}A-\epsilon $ by a simple curve of length $<2\,diam(D)+\epsilon $.

One can obtain a Riemannian ball from a graph $\Gamma $ by thickening its edges-so all edges become say solid cylinders 
of volume and length 1. One can, say, replace vertices by balls and obtain a handlebody. From this one obtains a ball $B$ of
approximately the same volume by gluing thickened disks on an appropriate set of simple closed curves on the handlebody.
By picking $\Gamma $ to be an expander graph (more precisely a sequence of expander graphs) and remarking that a disk
splitting $B$ induces a splitting of $\Gamma $ one see that $h$ has to be at least linear on $v$. This contrasts with the
$2$-dimensional case where the dependence on area $A$ is of the order of $\sqrt A$. Of course it is not clear whether $h$ exists at all.


\begin{thebibliography}{99}

\bibitem{Be}  A. S. Besicovitch,{\em On two problems of Loewner},
J. London Math. Soc. 27, (1952). 141--144.

\bibitem{BS} F. Balacheff, S. Sabourau, {\em Diastolic and isoperimetric inequalities on surfaces}, Ann. Sci.
Ecole Norm. Sup. 43 (2010) 579-605.

\bibitem{FK} S. Frankel, M. Katz, {\em The Morse landscape of a Riemannian disc }, Annales de l' Inst. Fourier
43 (1993), no. 2, 503-507.

\bibitem{Gr} M.Gromov, {\em Asymptotic invariants of infinite groups}, in "Geometric Group Theory", v.
2, 1-295, London Math. Soc. Lecture Note Ser., vol. 182, Cambridge Univ. Press, Cambridge,
1993.
\bibitem{Li} Y. Liokumovich, {\em  Surfaces of small diameter with large width}, preprint, math arXiv:1307.2306.

\bibitem{LNR} Y. Liokumovich, A. Nabutovsky, R. Rotman {\em Contracting the boundary of a Riemannian 2-disc}, preprint, math arXiv:1205.5474.

\bibitem{Pa} P.Papasoglu, {\em Cheeger constants of surfaces and isoperimetric inequalities}, Trans. Amer.Math.
Soc 361 (2009), no. 10, 5139-5162.
\end{thebibliography}
\end{document}
\bye